\RequirePackage{fix-cm}
\documentclass{svjour3}                     
\smartqed  
\usepackage{graphicx}
\usepackage{amsfonts}
\usepackage{epsfig, graphicx}
\usepackage{subfigure}
\usepackage{color}
\usepackage{longtable}
\usepackage{enumerate}
\usepackage{hyperref}
\usepackage{amsmath}

\newtheorem{thm}{Theorem}

\newcommand{\md}{\:\mbox{d}}
\newcommand{\FM}{\mbox{\tiny{FM}}}
\newcommand{\CR}{\mbox{\tiny{CR}}}
\newcommand{\pf}{\noindent{\textbf{Proof.}}\quad}

\begin{document}

\title{Optimal estimation for the Fujino--Morley interpolation error constants}


\author{Shih-Kang Liao, Yu-Chen Shu, Xuefeng Liu}


\institute{Shih-Kang Liao \at Graduate School of Science and Technology, Niigata University, Japan; Department of Applied
Mathematics, National Cheng Kung University, Tainan, Taiwan \\
              \email{L18021010@mail.ncku.edu.tw}           
           \and \\
           Yu-Chen Shu\at Department of Mathematics, National Cheng Kung University, Tainan, Taiwan \\
            \email{ycshu@mail.ncku.edu.tw}
            \and \\
            Xuefeng Liu (corresponding author) \at Graduate School of Science and Technology, Niigata University, Japan \\
            \email{xfliu@math.sc.niigata-u.ac.jp}
}

\date{Received: date / Accepted: date}

\maketitle

\begin{abstract}
    The quantitative estimation for the interpolation error constants of the Fujino--Morley interpolation operator is considered.
    To give concrete upper bounds for the constants, which is reduced to the problem of 
    providing lower bounds for eigenvalues of bi-harmonic operators, 
    a new algorithm based on the finite element method along with verified computation is proposed. 
    In addition, the quantitative analysis for the variation of eigenvalues upon the perturbation of the shape of triangles is provided.
    Particularly, for triangles with longest edge length less than one, the optimal estimation for the constants is provided.
    An online demo with source codes of the constants calculation is available at 
    \url{http://www.xfliu.org/onlinelab/}.
\end{abstract}

\section{Introduction}
\label{header-n457}
 The Fujino--Morley\footnote{The element discussed here is often called by ``Morley element" in the existing literature. However, the same element is also proposed independently by T. Fujino in P.739 of \cite{Fujino1971} (proceedings of a conference on 1969), which is also cited by L.S.D. Morley in \cite{Morley1971}.}  finite element method (FEM) \cite{Fujino1971,Morley1968,Morley1971} provides a robust way to solve partial differential problems evolving bi-harmonic operators.
 Especially, in solving the eigenvalue problem of bi-harmonic operators, the Fujino--Morley FEM along with the Fujino--Morley interpolation operator 
 can be utilized to find explicit lower bound for the eigenvalues; see the work of Carstensen--Gallistl \cite{Carstensen2013} and Liu \cite{Liu2015AMC}.
 For the Fujino--Morley interpolation, there are two fundamental constants that are playing important roles in bounding the eigenvalues. 
 In \cite{liu-you-amc-2018}, such constants are used to estimate the interpolation error constant for the quadratic Lagrange interpolation operator.
 Rough bounds of the two constants have been given in \cite{Carstensen2013} by using theoretical analysis.
In this paper, we will propose a FEM based method to provide the optimal estimation of the constants. 
 \\
 
  Let \(K\) be a triangle element with the largest edge length as $h$. The vertices of $K$ are denoted by \(O\), \(A\) and \(B\) and
  the edges by $e_1, e_2, e_3$; see Fig. \ref{fig:triangle}.

 \begin{figure}[htbp]
 \begin{center}
 \includegraphics[scale=0.18]{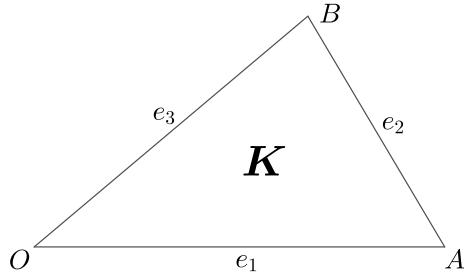}
 \caption{\label{fig:triangle}Triangle element $K$}
 \end{center}
 \end{figure}  

  Over $K$, we define the Fujino--Morley interpolation operator  $\Pi^{\FM}$. 
  For given $u\in H^2(K)$, $\Pi^{\FM}u$ is a quadratic polynomial that satisfies
  \begin{eqnarray}
    &&  (\Pi^{\FM}u -u)(P)=0, \quad P=O,A,B;\quad \\
    &&  \int_{e_i} \frac{\partial }{\partial n}(\Pi^{\FM}u -u) \md s=0, \quad i=1,2,3\:.  
  \end{eqnarray}
  
  For such an interpolation operator, the following error estimation is needed in numerical analysis, especially in solving the problem of bounding eigenvalues for bi-harmonic operators \cite{Carstensen2013,Liu2015AMC}; see, also, Theorem \ref{Liu_thm} for details of lower eigenvalue bounds evaluation:
  $$
  \|\Pi^{\FM}u -u\|_K \le C_0(K) |\Pi^{\FM}u -u|_{2,K},\quad
  |\Pi^{\FM}u -u|_{1,K} \le C_1(K) |\Pi^{\FM}u -u|_{2,K}\:.
  $$
The definition of constants $C_0(K)$ and $C_1(K)$ can be found in (\ref{eq:def-constants-0-1}).
In \cite{carstensen2014guaranteed,Carstensen2013}, to bound $C_1(K)$, the Crouzeix-Raviart constant $C_{\CR}$ is introduced. Using the fact $C_1(K)\le C_{\CR}(K)$, the rough bounds for the constants are given by
\begin{equation}
\label{eq:constant_est_existing}
 C_0(K) \le 0.2575 h^2 , \quad C_1(K) \le C_{\CR}(K)\le 0.2983 h\:.
\end{equation}
In \cite{Liu2015AMC}, the optimal estimation of $C_{\CR}(K)$ is given as
\begin{equation}
\label{eq:constant_est_existing_liu}
(C_1(K) \le) C_{\CR}(K)\le 0.1893h \:.
\end{equation}
%
%
%
 In this paper, we provide an algorithm to evaluate the constants directly and verify that the following estimation of the above two constants holds for element $K$ of arbitrary shapes.
 $$
 0.07349h^2\le C_0(K) \le 0.07353 h^2 ,\quad 0.18863 h \le C_1(K) \le 0.18868 h \:.
 $$
The lower bounds of these constants means there exists element $K$ such that $C_i(K)$ cannot be smaller than the lower bound provided here. To have rigorous computation results, the INTLAB toolbox 
of interval arithmetic \cite{Ru99a}  is adopted for the evaluation of the constants. 
Particularly, the rigorous eigenvalue estimation for matrices is based on the algorithm of Behnke \cite{behnke1991calculation}.\\

The method to be proposed in this paper for estimating constant $C_0$ and $C_1$ can also be used to estimate the constants of other interpolation operators. For example, let 
$\Pi$ be the Lagrange interpolation operator or the Fujino--Morley interpolation operator defined over triangle $K$, the following interpolation error estimation holds.
   $$
  \|\Pi u -u\|_{0,K} \le d_0 |u|_{2,K},\quad
  |\Pi u -u|_{1,K} \le d_1 |u|_{2,K}\:,
  $$
In \S\ref{header-n842}, sharp evaluation of constants $d_0$ and $d_1$ is provided; see detailed results in Table \ref{tab:LB_1} and Table \ref{tab:LB_2}. \\

 The rest of this paper is arranged as follows.  In \S\ref{header-n471}, constants $C_0$ and $C_1$ and the function space setting are introduced. 
 The method to solve eigenvalue problems corresponding to constants is presented in \S\ref{header-n658}. 
 The theoretical analysis about the perturbation of eigenvalues on element shape variation is performed in \S\ref{header-n516}. 
 In \S\ref{header-n753}, the algorithm to bound the constants for elements of arbitrary shapes is proposed and the optimal estimation is obtained. 
 In \S\ref{header-n842}, the optimal bound for $C_0$ and $C_1$ is applied in bounding error constants for other interpolation operators.

\section{Preliminary}\label{header-n471}
\paragraph{Function spaces} The standard notation for Sobolev space is used in this paper. 
That is, $\|\cdot \|_{\Omega}$ denotes the $L^2$ norm for $L^2$ space; $|\cdot|_{k,\Omega}$ $(k=1,2,\cdots)$ denotes the $k$th order semi-norm for functions in 
$H^k(\Omega)$. In many cases, the subscript $\Omega$ will be omitted if the domain is self-evident. 
The gradient operator is denoted by $\nabla$ and the second order derivative is given by $D^2u := (u_{xx}, u_{xy}, u_{yx},u_{yy})$ for $u\in H^2(\Omega)$.
The inner product of $L^2(\Omega)$ or $\left(L^2(\Omega)\right)^2$ is denoted by $(\cdot, \cdot)$.

\paragraph{Definition of constants}
To give the definition of the interpolation constants $C_0$ and $C_1$, 
let us introduce the kernel space of $\Pi^{\FM}$,,denoted by $V^{\FM}(K)$, when the operator is applied to a triangle element $K$.  That is,
\begin{equation}
\label{def:V-FM}
V^{\FM}(K):=\{u\in H^2(K) \:|\: u(O)=u(A)=u(B)=0,
\int_{e_i}\frac{\partial u}{\partial n}ds=0,~i=1,2,3 \}\:.
\end{equation}
Over space $V^{\FM}(K)$, the constants are defined by using the Rayleigh quotient.
\begin{equation}
\label{eq:def-constants-0-1}
     C_0(K):=\sup_{\scriptsize{\begin{array}{c}u\in V^{\FM}(K)\\D^2u\neq0\end{array}}}\frac{\|u\|_K}{\|D^2u\|_K}, ~~~
     C_1(K):=\sup_{\scriptsize{\begin{array}{c}u\in V^{\FM}(K)\\D^2u\neq0\end{array}}}\frac{\|\nabla u\|_K}{\|D^2u\|_K}\:.  
\end{equation}
Given a reference triangle $K$ with $\text{diam}(K)=1$, let $K_h$ be the triangle obtained through scaling $K$ by $h$ times, that is, $\text{diam}(K_h)=h$. It is easy to see that 
$C_0(K_h)=h^2C_0(K)$, $C_1(K_h)=hC_1(K)$. \\

Below, we also introduce the Crouzeix-Raviart interpolation constant, which will help to find the optimal estimation of the constant $C_1$.

\paragraph{Crouzeix-Raviart interpolation constant}

  Given $u\in H^1(K)$, the Crouzeix-Raviart interpolation $\Pi^{\CR}u$ is a linear polynomial over $K$ such that
    \begin{equation}
    \int_{e_i} (\Pi^{\CR}u - u ) \md s =0, \quad i=1,2,3.
    \end{equation}
    The Crouzeix-Raviart constant associated to $\Pi^{\CR}$ is defined as follows,
  \begin{equation}
  \label{eq:def-CR-constant} 
  C_{\CR}(K)=\sup_{\scriptsize{\begin{array}{c}
      u\in H^1(K)   \\
      \nabla (u -\Pi^{^{\CR}}u)\neq0  
  \end{array}}}\frac{\|u-\Pi^{^{\CR}}u \|}{\|\nabla (u -\Pi^{^{\CR}}u) \|} = \sup_{\scriptsize{\begin{array}{c}
       u\in V^{\CR}(K)  \\
       \nabla u\neq0 
  \end{array}}}\frac{\|u\|}{\|\nabla u\|} \:,
  \end{equation}
  where
  $$
    V^{\CR}(K):=\{u \in H^1(K)\:|\: \int_{e_i}u \md s=0, i=1,2,3\}\:. 
  $$
  The Crouzeix-Raviart interpolation constant $C_{\CR}$ is well investigated in \cite{Liu2015AMC}:
  \begin{enumerate}
  \item [a)] When vertices $O$ and $A$ of $K$ are fixed, the value of \(C_{\CR}(K)\) has monotonicity upon the \(y-\)coordinate of \(B\).
  \item [b)] 
  For all triangles with diameter less than $1$, the maximum value of $C_{\CR}(K)$ has a rigorous bound as follows, 
  $$
    \max_{\mbox{\footnotesize{diam}}(K)\le 1} C_{\CR}(K) \in [0.1890,0.1893] \:.
  $$
  \item [c)] Numerical computation implies the maximum value is achieved when \(K\) is a regular triangle. 
  \end{enumerate}

For the relation between $C_1(K)$ and $C_{\CR}(K)$, we have the following lemma (see, e.g., \cite{Carstensen2013}).
  
\begin{lemma}\label{lemma1}
 \(C_1(K)\leq C_{\CR}(K)\) for all triangle $K$.
\end{lemma}
\pf One can draw the conclusion by noticing that $u_x , u_y \in V^{\CR}(K)$ for any \(u \in V^{\FM}(K)\) and 
the inequalities $\|u_x\|\leq C_{\CR}(K)\|\nabla u_x\|$, $\|u_y\|\leq C_{\CR}(K)\|\nabla u_y\|$.\\

The task to obtain the optimal estimation of constants can be divided into two steps.

  \begin{enumerate} 
  \item [Step 1.] Direct evaluation for \(C_0\) and \(C_1\) for several sample triangles by
    solving the corresponding eigenvalue problem along with the Fujino-Morley FEM; see the detail in \S \ref{header-n658}.
  \item [Step 2.] Perturbation analysis of constant \(C_0\) and \(C_1\) upon the
    change of shape of triangle element, where the monotonicity of constant \(C_0\) and \(C_{\CR}\) will play an important role; see the detail in \S \ref{header-n516} and \S \ref{header-n753}. \\
  \end{enumerate}

For the purpose of simplicity, the vertices of $K$ are located at $O(0,0)$, $A(1,0)$, $B(a,b)$, while $|OB|\le 1$ and $|AB|\le 1$.
Also, due to the symmetry of the position of $B$, only the case that $a\ge 1/2$ will be considered. As a summary, we will focus on  
triangles with $B$ inside the following area $\Omega$ (see Fig. \ref{fig:Omega}): 
$$
\Omega=\{(a,b)\in\mathbb R^2~|~a^2+b^2\le 1,~a\ge 1/2,~b>0\}\:.
$$
  
 \begin{figure}[htbp]
 \begin{center}
 \includegraphics[scale=0.3]{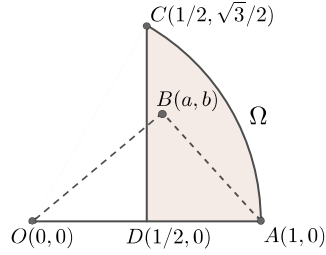}
 \caption{\label{fig:Omega}Range of vertex $B$}
 \end{center}
 \end{figure}

\section{Point-wise estimation of constants \texorpdfstring{\(C_0(K)\)}{C0(K)} and \texorpdfstring{\(C_1(K)\)}{C1(K)}} \label{header-n658}

In this section, let us describe the algorithm to estimate constants $C_0(K)$ and $C_1(K)$.
The interpolation constants are determined by solving eigenvalue problems of bi-harmonic operators, and their upper bounds will be evaluated by applying the Fujino-Morley FEM. 
The lower bounds of the interpolation constants will be discussed in \S\ref{header-n843}, where the conforming polynomial spaces are used.

The constants $C_0$ and $C_1$ are corresponding to the eigenvalues of the following eigenvalue problems.
\paragraph{Problem a)} Find $u \in V^{\FM}(K)(\subset H^2(K))$ and $\lambda>0$ such that

\begin{equation}
(D^2 u, D^2 v)  = \lambda (u,v)\quad \forall v\in V^{\FM}(K)
\end{equation}

\paragraph{Problem b)} Find $u \in V^{\FM}(K)(\subset H^2(K))$ and $\mu >0$ such that

\begin{equation}
(D^2 u, D^2 v)  = \mu (\nabla u,\nabla v)\quad \forall v \in V^{\FM}(K)
\end{equation}

The distribution of eigenpairs of Problem a) and Problem b) has been well investigated under
the theories for compact self-adjoint differential operators; see, e.g., \cite{Barbuska1991}. 
Let $\lambda_1$ and $\mu_1$ be the smallest eigenvalue of Problem a) and b), respectively.
Then, it is easy to see that 
$$
C_0=\sqrt{\lambda_1}^{-1}, \quad 
C_1=\sqrt{\mu_1}^{-1}.
$$

To give explicit bound of $\lambda_1$ and $\mu_1$, let introduce several finite element spaces.
Let $\mathcal{T}^h$ be a proper triangulation of domain $K$. A general Fujino--Morley finite element space $V^h$ has the member function $u_h$ with the following properties.

\begin{enumerate}
\item [a)] $u_h$  is piece-wise quadratic polynomial on each element of $\mathcal{T}^h$;
\item [b)] $u_h$ is continuous on each vertex;
\item [c)] $\displaystyle{\int_{e} \frac{\partial u}{\partial n} \md s}$ is continuous across each interior edge $e$.
\end{enumerate}
%
To approximate $V^{\FM}$ (see definition at (\ref{def:V-FM})), let us define a subspace of $V^h$ by introducing boundary conditions.
\begin{eqnarray*}
V^{\FM,h}(K)  :=\{u_h\in V^h ~|~  u_h(O)=u_h(A)=u_h(B)=0;\\
\int_{e_i} \frac{\partial u}{\partial n}\md s=0 \text{ for each $e_i$ of $K$}\:. \}\:.
\end{eqnarray*}
\vspace{-1.5cm}
\begin{eqnarray}\label{FM_femspace}
     \hfill 
\end{eqnarray}
\\

Next, let us define the approximate eigenvalue problems over $V^{\FM,h}(K)$.
Since the function in $V^{\FM,h}(K)$ may not have continuity across interior edges, the differential operators $\nabla$ and $D^2$ are 
piece-wisely defined on the triangulation $\mathcal{T}^h$.

\paragraph{Problem a')} Find $u \in V^{\FM,h}(K)$ and $\lambda_h>0$ such that

\begin{equation}
(D^2 u, D^2 v)  = \lambda_h (u,v) \in V^{\FM,h}(K)  
\end{equation}
\paragraph{Problem b')} Find $u \in V^{\FM,h}(K)$ and $\mu_h >0$ such that
\begin{equation}
(D^2 u, D^2 v)  = \mu_h (\nabla u,\nabla v) \in V^{\FM,h}  
\end{equation}
Denote the smallest eigenvalue of a') and b') by $\lambda_{h,1}$ and $\mu_{h,1}$, respectively. \\

Below, we quote Theorem 2.1 of \cite{Liu2015AMC} for the purpose of bounding eigenvalues, 
where the spaces and bilinear forms are taken as follows.
$$
V(h):=\{u+u_h| u \in V^{\FM}, u_h \in V^{\FM, h}\}, \quad  M(u,v):=(D^2u,D^2v)
$$
$$
N(u,v) := (u, v) \mbox{ for } C_0, \quad  N(u,v) := (\nabla u, \nabla v) \mbox{ for } C_1 \:.
$$
Particularly, $M(u,v)$ is an inner product for $V(h)$.

\begin{thm}\label{Liu_thm}
Let \(P_h:V(h) \mapsto V^{\FM, h}\) be the projection with
respect to inner product $(D^2 \cdot, D^2\cdot)$, i.e., for any \(u\in V(h)\)
\begin{equation}\label{M-projection}
(D^2(u-P_h u), D^2 v_h) = 0, \quad \forall v_h\in V^h.
\end{equation}
Suppose there exist quantities \(C_{h,0}\) and \(C_{h,1}\) such that for any $u \in V(h)$
\begin{equation*}\label{M-projection-estimate}
\|u - P_h u\| \leq C_{h,0} \|D^2(u - P_h u) \|, \quad 
\|\nabla(u - P_h u)\| \leq C_{h,1} \|D^2(u - P_h u) \|\:.
\end{equation*}
\end{thm}
Then we have,
\begin{equation}\label{framework lower bounds}
\lambda_1 \ge \frac{\lambda_{h,1}}{1+\lambda_{h,1}C_{h,0}^2}, \quad
\mu_1 \ge \frac{\mu_{h,1}}{1+\mu_{h,1}C_{h,1}^2}\:.
\end{equation}

\begin{remark}
The above theorem does not require \(V^h\subset V\).
Thus, we can use non-conforming finite element methods to obtain lower
eigenvalue bounds. 
\end{remark}

Due to the special setting of the nonconforming space $V^h$, the projection $P_h$ here is nothing else but the 
Fujino--Morley interpolation defined over $\mathcal{T}^h$. That is,
$$
(P_h u)|_T = \Pi^{\FM}(u|_T), \quad \mbox{ for each } T \in \mathcal{T}^h\:.
$$
The error constants $C_{0,h}$ and $C_{1,h}$ are given by 
$$
C_{i,h} = \max_{T \in \mathcal{T}^h} C_i(T),  \quad i=0,1 \: .
$$
From existing estimation of $C_0$ and $C_1$, as quoted in (\ref{eq:constant_est_existing}) and (\ref{eq:constant_est_existing_liu}), we have
$$
\|u-P_h u\|\leq0.2575h^2\|D^2(u- P_h u)\|, \quad \|\nabla(u-P_h u)\|\leq0.1893h\|D^2(u- P_h u)\| \:.
$$
Therefore, the explicit lower bounds of $\lambda_1$ and $\mu_1$ are given as 
\begin{equation}\label{framework lower bounds_1}
{\lambda_1}\geq\frac{{\lambda_{h,1}}}{1+0.2575^2h^4{\lambda_{h,1}}},\quad{\mu_1}\geq\frac{{\mu_{h,1}}}{1+0.1893^2h^2{\mu_{h,1}}}\:.
\end{equation}
\begin{remark}
The technique of K. Kobayashi in \cite{Kobayashi-2015-CM} can be applied here to provide upper bound of $C_0$ (lower bound of $\lambda_1$) directly without a prior information of $C_0$, 
i.e., the rough bound in (\ref{eq:constant_est_existing}) and (\ref{eq:constant_est_existing_liu}).
\end{remark}

\section{Variation of constants upon perturbation of triangle
shape}\label{header-n516}

In the last section, Theorem \ref{Liu_thm} provides a method to bound constants for a concrete triangle $K$. 
To bound constant $C_0$ and $C_1$ for triangles of arbitrary shapes, we need to consider the variation of constants upon the perturbation of triangle shape.
\paragraph{Linear perturbation of triangle $K$}
Define linear mapping \(Q\) by
\begin{equation}
\label{def:Q}
Q:=\left( \begin{array}{cc}1 &\alpha \\
0 &\beta  \end{array}\right),    
\end{equation}
where $\alpha\approx0$ and $\beta\approx1$.
Apply \(Q\) to the point \((x,y)^T\) of \(K\) by
\[\left(      \begin{array}{c} x \\y  \end{array} \right) \rightarrow \left(      \begin{array}{c} \tilde x \\\tilde y  \end{array} \right) =Q\left(      \begin{array}{c} x \\y  \end{array} \right)=\left(      \begin{array}{c} x+\alpha y \\\beta y  \end{array} \right)\]
and \(\widetilde{K}\) is the triangle with vertices \(O(0,0),A(1,0)\) and
\(\widetilde B(x+\alpha y,\beta y)\).
For the variation of norms of $u$ on $K$ under $Q$, we have the following lemma.
\begin{lemma}\label{lemma2} Given $u\in H^2(K)$ and define $\tilde{u}=u\circ Q^{-1}\in H^2(\widetilde{K})$, where $Q=\left( \begin{array}{cc}1 &\alpha \\
0 &\beta  \end{array}\right)$, and $\beta > 0$. Define $\gamma=\alpha^2+\beta^2+1$. Then
\begin{enumerate}[(a)]
\item 
  For \(L^2(K)\)-norm, we have
\[\|\tilde u\|^2_{\widetilde K}=\beta\|u\|^2_{K}\:.\]
\item 
  For \(H^1(K)\)-norm, we have
\[\displaystyle\frac{\gamma-\sqrt{\gamma^2-4\beta^2}}{2}\|\nabla\tilde  u\|^2_{\widetilde K}\leq \beta \|\nabla u\|^2_{K}\leq\frac{\gamma+\sqrt{\gamma^2-4\beta^2}}{2}\|\nabla\tilde u\|^2_{\widetilde K}\:.\]
\item 
  For the \(H^2(K)\)-norm, we have
\[
\displaystyle 
\frac{\left(\gamma-\sqrt{\gamma^2-4\beta^2}\right)^2}{4}\|D^2\tilde u\|^2_{\widetilde K}
\leq\beta \|D^2 u\|^2_{K}
\leq  \frac{\left( \gamma+\sqrt{\gamma^2-4\beta^2 }\right)^2}{4} \|D^2\tilde u\|^2_{\widetilde K}\:.
\]
\end{enumerate}
\end{lemma}
\pf The equality of (a) is evident. Let $K$ be the triangle with fixed vertices $O(0,0)$ and $A(1,0)$. Let us introduce the linear transform of $K$. Since
\((u_x,u_y)^t=Q\cdot(\tilde u_{\tilde x},\tilde u_{\tilde y})^t\), we have \begin{equation*}
\lambda_{\min}(Q^tQ)\cdot(\tilde u_{\tilde x}^2+\tilde u_{\tilde y}^2)\le u_x^2+u_y^2\le\lambda_{\max}(Q^tQ)\cdot(\tilde u_{\tilde x}^2+\tilde u_{ y}^2)
\end{equation*}
where \(\lambda_{\min}(Q^tQ)\) and \(\lambda_{\max}(Q^tQ)\) denoted the minimum and maximum eigenvalues of \(Q^tQ\), respectively. Therefore
\[\displaystyle {\lambda_{\min}(Q^tQ)}\|\nabla\tilde  u\|^2_{
\widetilde K}\leq {\beta} \|\nabla u\|^2_{K}\leq {\lambda_{\max}(Q^tQ)} \|\nabla\tilde u\|^2_{\widetilde K}.\] 
The eigenvalues of $Q^t Q$ are listed below:
\[
\lambda_{\min}(Q^tQ)=\frac{\gamma-\sqrt{\gamma^2-4\beta^2}}{2};\quad
\lambda_{\max}(Q^tQ)=\frac{\gamma+\sqrt{\gamma^2-4\beta^2}}{2},\]
where $\gamma=\alpha^2+\beta^2+1$.

For the second order derivatives, we have \[(u_{xx},u_{xy},u_{yx},u_{yy})^t=T\cdot(\tilde u_{\tilde x\tilde x},\tilde u_{\tilde x\tilde y},\tilde u_{\tilde y\tilde x}, \tilde u_{\tilde y\tilde y})^t,\]
where
\[
T=\left( \begin{array}{cccc}1 &0 &0 &0 \\
\alpha &\beta &0 &0 \\
\alpha &0 &\beta &0 \\
\alpha^2 &\alpha\beta &\alpha\beta &\beta^2
\end{array}\right).
\]
We denote the minimum and maximum eigenvalues of \(T^tT\) by \(\lambda_{\min}(T^tT)\) and \(\lambda_{\max}(T^tT)\), respectively. 
For the \(H^2(K)\) norm, we have
\[
{\lambda_{\min}(T^tT)}\|D^2\tilde  u\|^2_{\widetilde K}\leq {\beta} \|D^2 u\|^2_{K} 
\leq \lambda_{\max}(T^tT)\|D^2\tilde u\|^2_{\widetilde K},
\] 
where the minimum and maximum eigenvalues are
\begin{eqnarray*}
\lambda_{\min}(T^tT)&=&\frac{\gamma^2-2\beta^2-\gamma\sqrt{\gamma^2-4\beta^2}}{2}=\lambda^2_{\min}(Q^tQ),\\
\lambda_{\max}(T^tT)&=&\frac{\gamma^2-2\beta^2+\gamma\sqrt{\gamma^2-4\beta^2}}{2}=\lambda^2_{\max}(Q^tQ)\:.
\end{eqnarray*}
\hfill\(\square\)\\
\begin{remark}
Note that there are repeated eigenvalue $\beta^2$ of positive definite matrix $T^tT$ with algebraic multiplicity 2. 
It's easily to verify $\lambda_{\max}(T^tT)>\beta^2>\lambda_{\min}(T^tT)>0$ by using the fact $\gamma^2-4\beta^2=(\gamma-2\beta^2)^2+4\alpha^2\beta^2>0.$  
\end{remark}

Below, we show several results for special values of $\alpha$ and $\beta$.

\begin{lemma}\label{lemma2a}
Let $\alpha = 0$, $\beta = 1+\epsilon$. For $\epsilon > 0$, we have
\[
(1+\epsilon)^{-1}\|\nabla\tilde  u\|^2_{\widetilde K} \leq\|\nabla u\|^2_{K} \leq (1+\epsilon)
\|\nabla\tilde u\|^2_{\widetilde K}\:,
\]
and 
\[\displaystyle 
(1+\epsilon)^{-1}
\|D^2\tilde u\|^2_{\widetilde K}
\leq\|D^2 u\|^2_{K}
\leq (1+\epsilon)^3
\|D^2\tilde u\|^2_{\widetilde K}\:.\]
For $-1<\epsilon<0$, we have
\[\displaystyle 
(1+\epsilon)\|\nabla\tilde  u\|^2_{\widetilde K} \leq\|\nabla u\|^2_{K} \leq (1+\epsilon)^{-1}
\|\nabla\tilde u\|^2_{\widetilde K}\:,
\]
and 
\[\displaystyle 
(1+\epsilon)^{3}
\|D^2\tilde u\|^2_{\widetilde K}
\leq\|D^2 u\|^2_{K}
\leq (1+\epsilon)^{-1}
\|D^2\tilde u\|^2_{\widetilde K}\:.\]
\end{lemma}

Next lemma will be used in case of the perturbation of $B$ along the arc with center as $O$ and radius as $|AB|$.

\begin{lemma}\label{lemma2c}
For $0<\theta<\pi$, $0<\theta+\tau<\pi$, define 
linear mapping $Q$ with the following 
$\alpha$ and $\beta$,
$$
\alpha:=\frac{\cos(\theta+\tau)-\cos\theta}{\sin\theta}, \quad
\beta:=\frac{\sin(\theta+\tau)}{\sin\theta} \:.
$$
Then, in case $\tau<0$,
\begin{eqnarray}
\frac{\eta}{\sqrt{\beta}}
\|\nabla\tilde  u\|_{\widetilde K} 
\leq\|\nabla u\|_{K} \leq 
\frac{\rho}{\sqrt{\beta}} \|\nabla\tilde u\|_{\widetilde K},\\
\frac{\eta^2}{\sqrt{\beta}}
\|D^2\tilde u\|_{\widetilde K}
\leq\|D^2 u\|_{K}
\leq\frac{\rho^2}{\sqrt{\beta}}
\|D^2\tilde u\|_{\widetilde K}\:,
\end{eqnarray}
where 
$\rho$ and $\eta$ are defined  by
$$
\rho:=\frac{\cos\left(\frac{\theta+\tau}{2}\right)}{ \cos \left(\frac{\theta}{2}\right)},\quad
\eta:=\frac{\sin\left(\frac{\theta+\tau}{2}\right)}{ \sin \left(\frac{\theta}{2}\right)}\:.
$$
In case $\tau>0$, the above inequalities hold by exchanging the value of $\rho$ and $\eta$.

\end{lemma}
\begin{proof}
This lemma is a direct result of Lemma \ref{lemma2} by noticing the following relations.
$$
\gamma=\frac{2(1-\cos\theta\cos(\theta+\tau))}{\sin^2\theta},\quad \gamma^2-4\beta^2=\frac{4\alpha^2}{\sin^2\theta}\:.
$$
\end{proof}

\subsection{Properties of constant \texorpdfstring{$C_0(K)$}{C0(K)}}\label{header-n567}


\begin{thm}\label{C_0_1} Let $K$ be the triangle with vertices $O(0,0)$, $A(1,0)$, $B(a,b) \in \Omega$, then
\(C_0(K)\) monotonically increases as $b$ increases.
\end{thm}
\begin{proof}
Take $\alpha=0$, $\beta=1+\epsilon$ ($\epsilon>0$) for transformation $Q$ in (\ref{def:Q}). 
Then $QB$ moves $B$ along y-direction and
\[
\|\tilde{v}\|^2_{\widetilde{K}}=(1+\epsilon)\|v\|^2_{K}, \quad \|D^2\tilde{v}\|^2_{\widetilde{K}}  \leq(1+\epsilon)\|D^2v\|^2_{K}\:.
\]
Then we can easily draw the conclusion from the definition of the constant.
\end{proof}


Below we consider the perturbation of \(B=(\cos\theta,\sin\theta)\) along $\theta$ direction. 
By using Lemma \ref{lemma2} and \ref{lemma2c}, we can easily obtain the following result.
\begin{thm}\label{C_0_2} 
Let $K$ be the triangle with vertices $O(0,0)$, $A(1,0)$, $B(\cos\theta,\sin\theta)$ ($B \in \Omega$)
and $\widetilde{K}$ be the triangle with vertices $O$, $A$ and
\(\widetilde{B}=(\cos(\theta+\tau),\sin(\theta+\tau))\) ($\widetilde{B} \in \Omega$). 
We have the following two inequalities: 
\[
C_0(\widetilde{K})\leq\frac{\cos^2(\frac{\theta+\tau}{2})}{\cos^2(\frac{\theta}{2})}C_0(K)~~ (\tau<0);~~
C_0(\widetilde{K})\leq\frac{\sin^2(\frac{\theta+\tau}{2})}{\sin^2(\frac{\theta}{2})}C_0(K)~~ (\tau>0)\:.
\]
\end{thm}

\subsection{Properties of \texorpdfstring{$C_{\CR}(K)$}{CCR(K)} and \texorpdfstring{\(C_1(K)\)}{C1(K)}}\label{header-n609}
As shown in (\ref{eq:constant_est_existing_liu}), an upper bound of $C_1(K)$ is known already via the estimation of $C_{\CR}(K)$. 
For the purpose of a more accurate estimation for \(C_1(K)\), the perturbation of  $C_1(K)$ respect to $B$ will be required in 
next section. \\




As a preparation, let us consider the perturbation analysis of $C_{\CR}(K)$ upon moving $B$ along the $x-$direction. 

\begin{thm}\label{C_1_1} Let $K$ be the triangle with vertices $O(0,0)$, $A(1,0)$, $B(a,b)$ ($B \in \Omega$) and $\widetilde{K}$ be the triangle with vertices
\(O\), \(A\) and \(\widetilde{B}(a+b\epsilon,b)\) ($\widetilde{B} \in \Omega$) with the condition \(|\epsilon|<1/2\). Then,
\begin{eqnarray*}
C_{\CR}(\widetilde{K}) \leq \left(1+\frac{|\epsilon|}{2}+\frac{3\epsilon^2}{8}\right)~ C_{\CR}(K)~~ 
\end{eqnarray*}
\end{thm}
\begin{proof}
Take $\alpha=\epsilon$, $\beta=1$ for transformation $Q$ in (\ref{def:Q}). Since 
for \(\tilde{v}\in V^{\CR}(\widetilde{K})\), \(v=\tilde{v}\circ Q\in V^{\CR}(K)\),
from Lemma \ref{lemma2} , we have
$$
\|\tilde{v}\|^2_{\widetilde{K}}=\|v\|^2_{K}, ~~
\|\nabla\tilde v\|^2_{\widetilde K}\geq \frac{2}{\epsilon^2+2+\sqrt{\epsilon^4+4\epsilon^2}}\|\nabla v\|^2_{K}\:.
$$
By using Eq. \eqref{eq:def-CR-constant} and the above two equations, we have
\begin{equation}
\label{eq:estimate_along_x}    
C_{\CR}(\widetilde K)\leq \sqrt{\frac{\epsilon^2+2+\sqrt{\epsilon^4+4\epsilon^2}}{2}}C_{\CR}(K).
\end{equation}

In case \(\epsilon>0\), we have
\begin{eqnarray*} \frac{\epsilon^2+2+\sqrt{\epsilon^4+4\epsilon^2}}{2} 
\leq \frac{2+\epsilon^2+\epsilon(2+\epsilon)}{2}
\leq \left(1+\frac{\epsilon}{2}+\frac{3\epsilon^2}{8}\right)^2,
\end{eqnarray*}
and in case of \(\epsilon<0\) and $|\epsilon| < 1/2$, we have
\begin{eqnarray*} \frac{\epsilon^2+2+\sqrt{\epsilon^4+4\epsilon^2}}{2} 
\leq \frac{2+\epsilon^2-\epsilon(2-\epsilon)}{2}
\leq \left(1-\frac{\epsilon}{2}+\frac{3\epsilon^2}{8}\right)^2\:.
\end{eqnarray*}
Thus we can draw the conclusion.
\end{proof}

\begin{remark}
In the above proof, we use polynomial of $\epsilon$ to simplify the expression in (\ref{eq:estimate_along_x}).
The estimation of the constant has the exact order up to linear term \(\epsilon\), 
while the coefficients of quadratic term \(\epsilon^2\) are overestimated for the purpose of a simple expression.
\end{remark} 
\begin{thm}\label{C_1_2}
Let $K$ be the triangle with vertices $O(0,0)$, $A(1,0)$, $B(a,b)$  ($B \in \Omega$) and
\(\widetilde{K}\) be the triangle with vertices \(O(0,0)\), \(A(1,0)\) and
\(\widetilde{B}(a+\alpha b,(1+\epsilon) b)\)  ($\widetilde{B} \in \Omega$).
Assume \(|\alpha|, |\epsilon| \le 1/2 \), we have
\begin{equation}
\label{eq:perb-est-1}
C_1(\widetilde{K})\leq  \left( 1 + \frac{3}{2} |\alpha| + 2 |\epsilon| + \frac{3}{2}|\alpha|^2 + |\epsilon|^2\right) C_1(K) \quad (\epsilon \geq 0)
\end{equation}
\begin{equation}
\label{eq:perb-est-2}
C_1(\widetilde{K})\leq  \left( 1 + \frac{3}{2} |\alpha| + |\epsilon| + \frac{3}{2}|\alpha|^2 + 2|\epsilon|^2\right) C_1(K) \quad (\epsilon \le 0)
\end{equation}
\end{thm}
\begin{proof}
Let 
\(Q=\left( \begin{array}{cc}1 &\alpha \\
0 & 1+\epsilon  \end{array}\right) \) be the linear mapping from $K$ to $\widetilde K$.
Usually,  $\partial u/\partial n$ varies under the transform $Q$. However, for \({v}\in V^{\FM}({K})\), 
we still have \(\tilde v={v}\circ Q^{-1}\in V^{\FM}(\widetilde{K})\). To see this, 
notice that the condition ${v}(O)={v}(A)={v}(B)=0$ implies $\int_{e_i}\partial v/\partial \tau \md s=0$, where $\tau$ is the tangent vector of $e_i$. 
From the condition $\int_{e_i} \partial {v}/\partial n \md s=0$, we have
$$
\int_{e_i} \nabla {v} \cdot (t_1\tau + t_2 n )\md s =0\quad  \quad (\forall t_1,t_2\in R)\:.
$$
where $n$ is the unit norm vector on $e_i$. Hence, 
$$
\int_{\tilde{e}_i} \frac{ \partial \tilde{v} }{ \partial \tilde{n} } \md s =\frac{1}{|\det(Q)|}\int_{e_i} (Q\cdot\nabla v^t)\cdot \tilde n \md s=0  \quad  (\tilde{n}: \text{ normal direction of edge } \tilde{e}_i )\:.
$$
From Lemma \ref{lemma2}, we have, with \(\gamma=\alpha^2+(1+\epsilon)^2+1\),
\begin{eqnarray*}
\|\nabla u\|^2_{K} &\ge & \frac{\gamma-\sqrt{\gamma^2-4(1+\epsilon)^2}}{2(1+\epsilon)}\|\nabla\tilde  u\|^2_{\widetilde K}\:,\\
\|D^2 u\|^2_{\widetilde K} &\le& \frac{\gamma^2-2(1+\epsilon)^2+\gamma\sqrt{\gamma^2-4(1+\epsilon)^2}}{2(1+\epsilon)}\|D^2\tilde u\|^2_{K}\:.\end{eqnarray*}
Define 
$$
f_1(\alpha,\epsilon) := \frac{\gamma-\sqrt{\gamma^2-4(1+\epsilon)^2}}{2(1+\epsilon)} 
,\quad
f_2(\alpha,\epsilon) := \frac{\gamma^2-2\beta^2+\gamma\sqrt{\gamma^2-4(1+\epsilon)^2}}{2(1+\epsilon)}\:.
$$
By taking the Taylor expansion of $\sqrt{f_2/f_1}$ at $\alpha=0,\epsilon=0$, we have, for $|\alpha|,|\epsilon|\le 1/2$, 
\begin{equation}
    \label{eq:taylor-1}
\frac{\sqrt{f_2}}{\sqrt{f_1}} \le 1 + \frac{3}{2} |\alpha| + 2 |\epsilon| +  \frac{3}{2}|\alpha|^2 + |\epsilon|^2 \quad (\epsilon \geq 0)\:,
\end{equation}
\begin{equation}
    \label{eq:taylor-2}
\frac{\sqrt{f_2}}{\sqrt{f_1}} \le 1 + \frac{3}{2} |\alpha| +  |\epsilon| +  \frac{3}{2}|\alpha|^2 + 2|\epsilon|^2 \quad (\epsilon \le 0)\:.
\end{equation}

By the definition of constant $C_1$, we can draw the conclusion.
\end{proof}

\begin{remark}
In the estimation (\ref{eq:taylor-1}) and (\ref{eq:taylor-2}), 
the coefficients of $|\alpha|$ and $|\epsilon|$ agree with the ones in Taylor expansion, while the coefficients of \(|\alpha|^2\) and $|\epsilon|^2$ are
over-estimated for the purpose of simple expression.
\end{remark}

\section{Optimal estimation of constants}\label{header-n753}
From the monotonicity of $C_0(K)$ on $y$-coordinate of $B$ as shown in Theorem \ref{C_0_1}, 
the maximum value of $C_0(K)$ can only happens on the arc $r=1$, $0<\theta\le\pi/3$. 
Since the value of $C_0(K)$ depends on the $x$- and $y$-coordinate of $B$, $C_0(K)$ can be regarded as a function on $x_B$, $y_B$. 
Fig. \ref{fig:C_0_contour} displays the contour lines of $C_0(K)$ respect to $x_B,y_B$ in $[0.5,1]\times (0,1]$. 
Numerical estimation of $C_{0}(K)$ implies that the maximum is taken when $K$ is the regular triangle. 
 \begin{figure}[htbp]
 \begin{center}
 \includegraphics[scale=0.3]{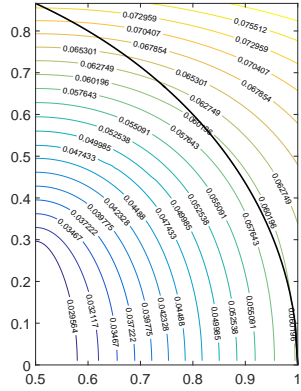}
 \caption{\label{fig:C_0_contour}The contour line of $C_0(K)$}
 \end{center}
 \end{figure}

\subsection{Optimal estimation of \texorpdfstring{\(C_0\)}{C0}}\label{header-n763}
On arc $r=1$, $0<\theta\le\pi/3$, we perform point-wise evaluation of $C_0(K)$ on a subdivision of $\theta$. 
Then the bound of $C_0(K)$ on whole arc is obtained by applying perturbation theory in Theorem \ref{C_0_2}.\\ 

Define \(\theta_i\) and \(\tau_i\) by
\[
\theta_i=\pi/3 \times \begin{cases}i \times 0.02&i=1,\ldots,48\\
0.95+0.05(1-2^{48-i}) & i=49,\ldots,59\\1&i=60. \end{cases}
\]
\[
\tau_i=\begin{cases}\theta_1 &i=1\\ \theta_i-\theta_{i-1} &i=2,\ldots,60.\end{cases}
\]
For each \(\theta_i\), we evaluate the constant \(C_0(K)\), then using the Theorem \ref{C_0_2} to give a upper bound of \(C_0(K)\) on each
sub-interval \((\theta_i-\tau_i,\theta_i]\) with perturbation
\(\tau_i.\) 
\begin{figure}[htbp]
 \begin{center}
 \includegraphics[scale=0.23]{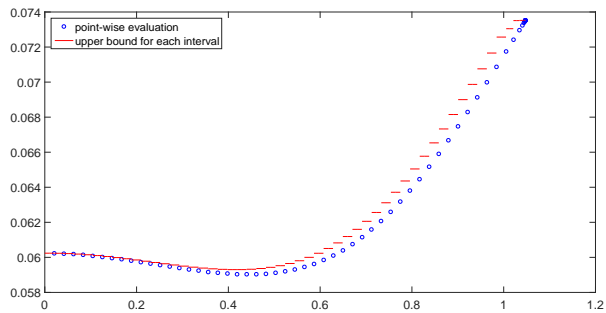}
 \caption{\label{fig:C_0_upperbound}Point-wise evolution of $C_0(K)$ at $\theta_i$ and upper bound of $C_0(K)$ for $\theta\in(\theta_i-\tau_i,\theta_i]$}
 \end{center}
 \end{figure} 
Fig. \ref{fig:C_0_upperbound} displays the upper bound of $C_0(K)$, where $x$-coordinate denote the size of $\angle BOA$ with range as $(0,\pi/3]$. 
The verified computing results show that $$C_0(K)\leq0.07353\:.$$ 

\subsection{Optimal estimation of \texorpdfstring{\(C_1\)}{C1}}\label{header-n778}

The estimation of $C_0(K)$ is relatively easily done, which is thanks to the monotonicity of the constant $C_0(K)$ with respect to the $y$-coordinate of vertex $B$ of $K$. 
However, such property of monotonicity is not available for $C_1(K)$. 
Thus, one has to consider the case of collapsed triangles (the vertex $B$ being close to $x$-axis), 
which is difficult to process because the estimation of Theorem \ref{C_1_2} 
has divergent bound for small $y$-coordinate of vertex $B(a,b)$.
To avoid such difficulties, we subdivide the area into two parts: 
$\Omega=\Omega_1\cup \Omega_2$ (see Fig. \ref{fig:Omega_rough_sepa}). 
On $\Omega_1$, we can estimate $C_1$ directly; 
on $\Omega_2$, the estimation of $C_1$ is done through $C_{\CR}(K)$ for collapsed triangles 
by using the relation $C_1(K) \le C_{\CR}(K)$. \\

The approximate numerical evaluation of $C_1(K)$ and $C_{\CR}(K)$ are displayed in Fig. \ref{fig:C_CR_C_1_contour},  from which we have the following information.

\begin{figure}[htbp]\begin{center}
     {\begin{tabular}[!]{cc}
   \subfigure{\includegraphics[height=65mm]{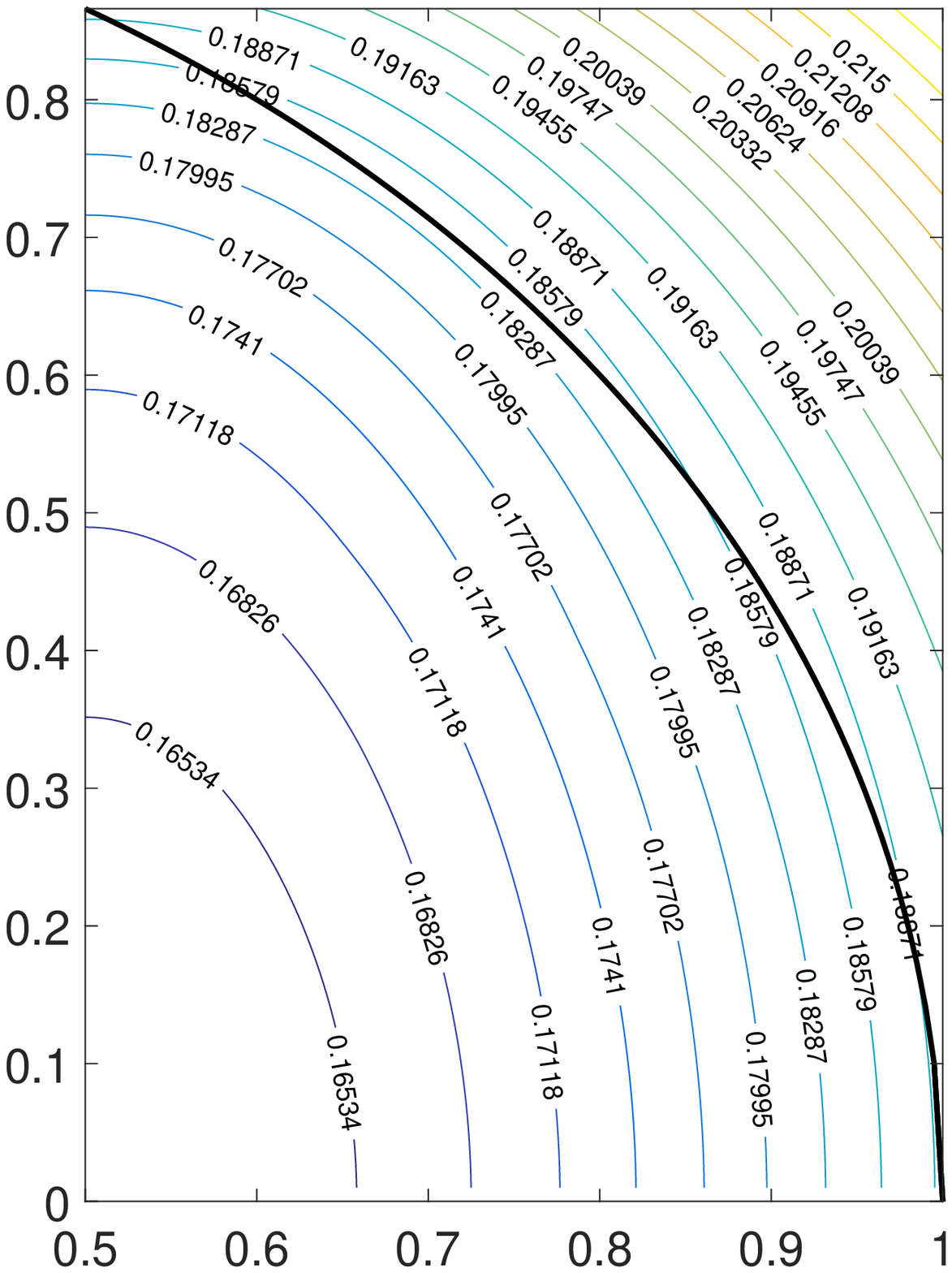}}
   
   \subfigure{\includegraphics[height=65mm]{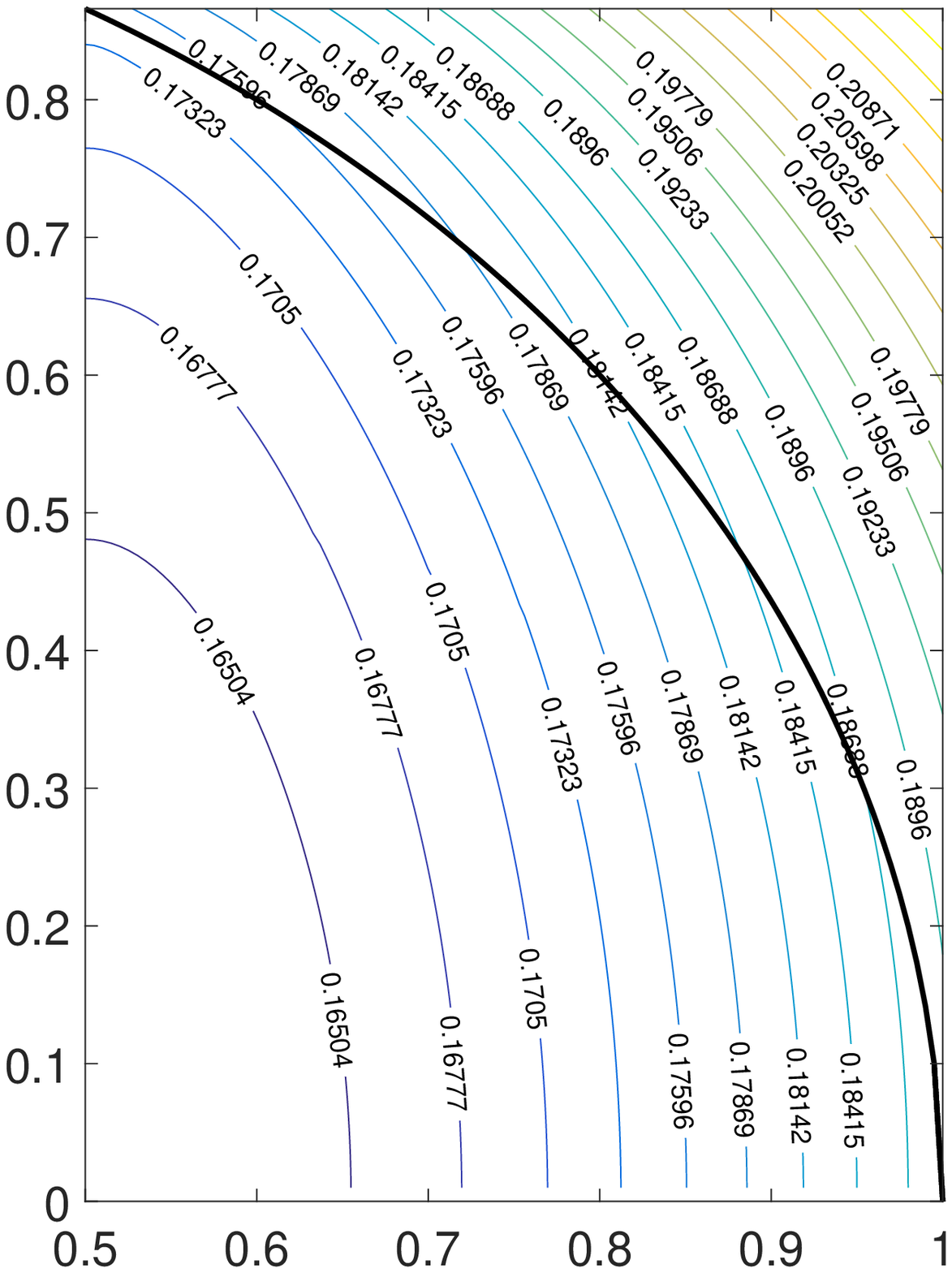}} 
       \end{tabular}}\end{center}
\caption{\label{fig:C_CR_C_1_contour}Contour lines of $C_{\CR}$(Left) and $C_1$(Right).}
\end{figure}

\begin{enumerate}
\item
  The maximum value of \(C_1(K)\approx 0.18868\) happens when
  \(B \) tends to \( (1,0)\) 

\item The maximum value of \(C_{\CR}(K)\approx0.18932\) happens at \(B=(1/2,\sqrt{3}/2)\).

\item \(C_{\CR}(K)\approx0.18868\) when $B$ tends to  $(1,0)$.

\item The constant $C_{\CR}(K)$ gives a nice bound (in the sense that the value of $C_{\CR}(K)$ is less than $0.18868$) for $C_1(K)$ when $B$ is below $y=\sin(0.98\pi/3)$.

\end{enumerate}

The above properties of $C_1(K)$ implies the optimal estimate value of \(C_1\) can be done by following the strategy below:

\begin{enumerate}
    \item [Step 1.] Separate the region $\Omega$ into parts $\Omega_1$ and $\Omega_2$ as follows(see Fig. \ref{fig:Omega_rough_sepa}):
              \[
              \Omega_1=\{(x,y)\in\Omega~|~ y\ge\sin(0.98\pi/3)\};\quad \Omega_2=\{(x,y)\in\Omega~|~y\le\sin(0.98\pi/3) \}.
              \]
        \begin{figure}[htbp]
        \begin{center}
        \includegraphics[scale=0.45]{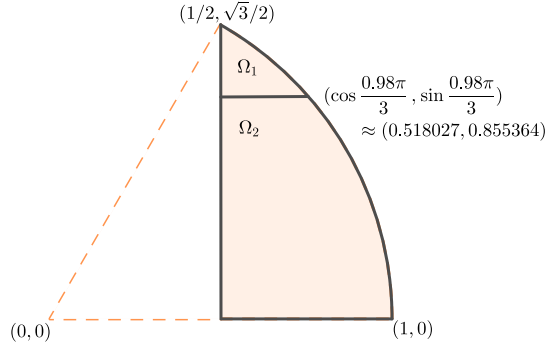}
        \caption{\label{fig:Omega_rough_sepa}Separating $\Omega$ into two parts $\Omega_1$ and $\Omega_2$}
        \end{center}
        \end{figure}
 
    \item [Step 2.] In \(\Omega_1\), by evaluating the constant $C_1$ for several position of \(B\) and apply the perturbation result in Theorem \ref{C_1_2}, we will have 
                         \begin{equation}\label{C_1_Omega1}
                          \sup_{B\in\Omega_1}C_1(K)\le 0.1847    \:.
                         \end{equation}
    \item [Step 3.] In $\Omega_2$, the upper bound of $C_1(K)$ will be obtained through $C_{\CR}(K)$,
                         \begin{equation}\label{C_CR_Omega2}
                         \sup_{B\in\Omega_2}C_1(K)\le\sup_{B\in\Omega_2}C_{\CR}(K)\le0.18868 \:.
                         \end{equation}
\end{enumerate}

Below, we show the details to obtain (\ref{C_1_Omega1}) and (\ref{C_CR_Omega2}).
\paragraph{Estimation of $C_1$ in $\Omega_1$} 
Let $K_0$ be the triangle with vertices at $(0,0)$, $(1,0)$, $\displaystyle B_0=(1/2,\sqrt{3}/2)$. 
Evaluation of upper bound of $C_1$ tells that $\displaystyle C_1(K_0) < 0.1744$. Notice that for all $x,y\in\Omega_1$, $|x-x_{B_0}|\le 0.021y_{B_0}$ and $|y-y_{B_0}|\le 0.013 y_{B_0}$. 
By applying estimation (\ref{eq:perb-est-2}) in Theorem \ref{C_1_2} with $\alpha=0.021$ and $\epsilon=0.013$, we have
\begin{equation}\label{eq:omega1_est}
\sup_{B\in\Omega_1}C_1(K) < 1.059\cdot C_1(K_0) < 0.1847\:.
\end{equation}

\paragraph{Estimation of $C_1$ in $\Omega_2$} Numerical results implies $C_1(K)$ and $C_{\CR}(K)$ have the same supremum on $\Omega_2$, 
which is reached when $B$ tends to $(1,0)$. Since $C_1(K)\le C_{\CR}(K)$ holds strictly, we just focus on $C_{\CR}(K)$. Due to the monotonicity of $C_{\CR}(K)$ on $y_B$, the maximum value of $C_{\CR}(K)$ is taken on either of the following two boundaries.
$$\Gamma_1=\Omega_1\cap\Omega_2;\quad \Gamma_2=\{(\cos\theta,\sin\theta)\in\mathbb{R}^2~|~0<\theta\le0.98\pi/3\}\:.$$

\begin{itemize}
    \item On $\Gamma_1$, we subdivide $\Gamma_1$ into small closed intervals:
$$
I_i=\left\{\left(x,\sin\frac{0.98\pi}{3}\right) \in \Gamma_1 ~|~ x \in \left[\frac{1}{2}+(i-1)h, \frac{1}{2}+ih \right]\right\},\quad i=1,\ldots,20\:. $$
Here, $\displaystyle{h=\frac{1}{20}\left(\cos\left(\frac{0.98\pi}{3}\right)-\frac{1}{2}\right)}$.

\begin{figure}[htbp]
 \begin{center}
 \includegraphics[scale=0.25]{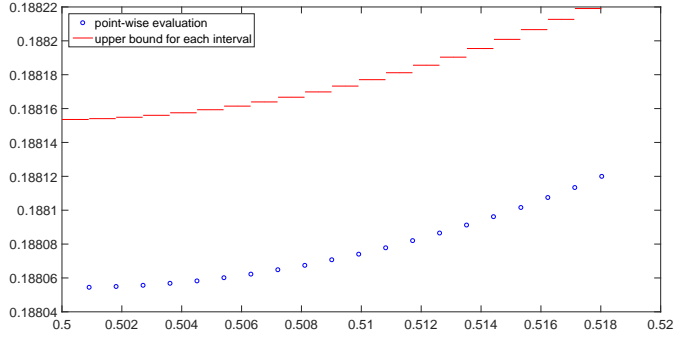}
 \caption{\label{fig:C_CR_upperbound_gamma1}Point-wise evaluation of $C_{\CR}(K)$ and upper bound of $C_{\CR}(K)$ on each sub-interval $I_i$.}
 \end{center}
 \end{figure}
The estimation of $C_{\CR}(K)$ on each $I_i$ is displayed in Fig. \ref{fig:C_CR_upperbound_gamma1}. 
The circles in Fig. \ref{fig:C_CR_upperbound_gamma1} denote the point-wise evaluation of $C_{\CR}(K)$ and the short bars denote the upper bound of $C_{\CR}(K)$ based on Theorem \ref{C_1_1}. 
The computation results tell that
\begin{equation}\label{eq:gamma1_est}
\sup_{B\in\Gamma_1}C_{\CR}(K)< 0.18822\:.
\end{equation}

\item To estimation of $C_{\CR}$ on $\Gamma_2$, we take the subdivision of $\theta$ as follows
$$\theta_i=0.01\times i\times\pi/3~.$$
Define $\tau_1=\theta_1$, $\tau_i=\theta_i-\theta_{i-1}\quad(i=2,\ldots,98)$. From the estimation in Theorem 4.2 of Liu\cite{Liu2015AMC}, 
we can estimate $C_{\CR}(K)$ for $\theta$ in each interval $(0,\theta_1]$ and $(\theta_{i-1},\theta_i]$ $(i=1,2,\ldots,98)$. 
In Fig. \ref{fig:C_CR_upperbound_gamma2}, the point wise estimation of $C_{\CR}(K)$ on $\theta_i$ and the upper bound of $C_{\CR}(K)$ for each interval are displayed. 
Computation results indicate that the supremum of $C_{\CR}(K)$ is reached when $B$ tends to $(1,0)$. We have the following strict bound of $C_{\CR}(K)$ 
\begin{eqnarray}\label{eq:gamma2_est}
\sup_{B\in\Gamma_2}C_{\CR}(K)< 0.18868\:.
\end{eqnarray}
\end{itemize}

\begin{figure}[htbp]
 \begin{center}
 \includegraphics[scale=0.25]{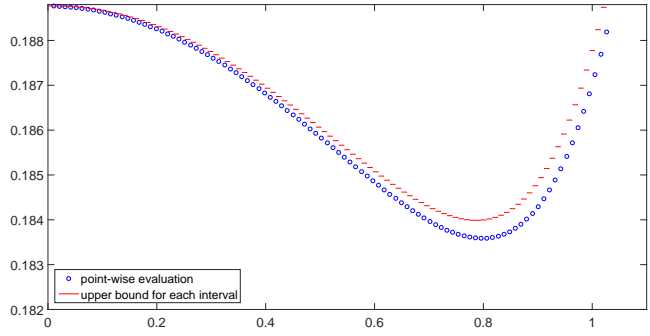}
 \caption{\label{fig:C_CR_upperbound_gamma2}Point-wise evaluation of $C_{\CR}(K)$ and upper bound of $C_{\CR}(K)$ on each sub-interval $(\theta_i-\tau_i,\theta_i)$.} 
 \end{center}
 \end{figure}

The estimation (\ref{eq:omega1_est}), (\ref{eq:gamma1_est}) and (\ref{eq:gamma2_est}) validate the results in (\ref{C_1_Omega1}) and (\ref{C_CR_Omega2}).  Thus, we can draw the conclusion that 
$$
\sup_{B\in\Omega}C_1(K) < 0.18868\:.
$$

\begin{remark}

The asymptotic value of $C_1$ when the vertex $B=(a,b)$ of $K$ tends to $(1,0)$ can be determined by theoretical analysis. 
Below is a sketch of determining $\lim_{\theta\to 0} C_1(K)$.  With the analogous argument as in Theorem 5 of \cite{LIU-Kikuchi}, 
the eigenvalue $\lambda$ corresponding to $C_1$ is determined by solving the following eigenvalue problem on one dimensional interval (0,1),
$$
xf_{xx} + f_x - \lambda x f =C, \quad  f(1)=0,\quad  \int_0^1 fdx=0\:.
$$
The general solution of this ODE  is given by utilizing hypergeometric function,
$$
f(x) = \widetilde{C}  J_0\left(\sqrt{\lambda}x\right) - C ~ _1F_2 \left(1;3/2,3/2; -\frac{\lambda x^2}{4}\right)~x\:.
$$
Here, $J_0$ is the $0$-th order Bessel function of the first kind and $~_1F_2, ~_2F_3$ are hypergeometric functions. 
By further applying the two boundary conditions, one can obtain the following equation with $\lambda$.
$$
_1F_2  \left(\frac{1}{2};1,\frac{3}{2};-\frac{x}{4}\right) \cdot ~ _1F_2  \left(1;\frac{3}{2},\frac{3}{2};-\frac{2}{4}\right) - \frac{1}{2} ~_2F_3  \left(1,1;\frac{3}{2},\frac{3}{2},2;-\frac{x}{4}\right) \cdot J_0\left(\sqrt{x}\right) =0
$$
The above equation has infinite solutions and simple computation tells that the smallest one, denoted by $\lambda_1$, is given by
$$
\lambda_1 = 28.10146739\cdots, \quad  1/\sqrt{\lambda_1} = 0.1886407440\cdots \:.
$$
Thus, $C_1(K)$ converges to $\displaystyle C_1=\frac{1}{\sqrt{\lambda_1}}=0.1886407440\cdots$ when $B$ tends to $(1,0)$. 
\end{remark}

\subsection{Lower bound of constants}\label{header-n843}
To confirm the precision of obtained estimation for $C_0$ and $C_1$, we also calculate the lower bounds of constants. 
The task to provide  lower bound is easy compared with the upper bound estimation. By evaluating $C_0(K)$ and $C_1(K)$ for a concrete element $K$ with conforming spaces, 
then we can have the lower bound for optimal constants.\\

Given a triangle $K$, the conforming finite dimensional space $W^m(\subset V_0)$ over $K$ can be constructed by using polynomials. 
$$
W^m:=\left\{ p \in P^m(K)\:|\: 
p(O)=p(A)=p(B)=0,
\int_{e_i}\frac{\partial p}{\partial n}\md s=0,\quad i=1,2,3.
\right\}\:.
$$
Here, $P^m(K)$ denotes the space of polynomial over $K$ with degree up to $m$.

For $C_0$, we choose $K$ as the unit regular triangle and solve Problem a) in $W^7$. Numerical computation tells that
$$C_0 \ge 0.073499\:.$$
Similarly, for $C_1$, by taking the triangle with $B=(\cos 0.01,\sin 0.01)$ and solving Problem b) in $W^6$, we have 
$$C_1 \ge 0.188638\:.$$

\section{Application to eigenvalue problems of Biharmonic operators}\label{header-n842}
Let us apply the two fundamental constants $C_0$ and $C_1$  to estimate the 
error constants appearing in  the Lagrange interpolation and the Fujino--Morley interpolation.\\

\paragraph{The Lagrange interpolation error constants} Given triangle element $K$ with vertices $P_1(0,0),$ $P_2(1,0)$  and $P_3(a,b)$. 
Let  $\Pi_1$ be the Lagrange interpolation over $T$ such that, for $v \in H^2(K)$,  $\Pi_1v$ is a linear polynomial  and $(v-\Pi_1v)(P_i)=0$, $i=1,2,3$. \\

Define subspace $V$ of $H^2(K)$ by
$$
V_0(K):=\{v\in H^2(K) \:|\: v(P_i)=0,\quad i=1,2,3\}\:.
$$
Define constants $C_{L,0}(K), C_{L,1}(K)$ by
$$
C_{L,0} (K):=\sup_{\scriptsize{\begin{array}{c}
    v\in H^2(K) \\
     D^2v\neq0 
\end{array}}}\frac{\|v-\Pi_1v\|}{\|D^2v\|}~(=\sup_{\scriptsize{\begin{array}{c}
     v\in V_0(K)  \\
     D^2v\neq0 
\end{array}}}\frac{\|v\|}{\|D^2v\|}),
$$
$$
 \displaystyle C_{L,1}(K):=\sup_{\scriptsize{\begin{array}{c}
    v\in H^2(K) \\
     D^2v\neq0 
\end{array}}}\frac{\| \nabla(v-\Pi_1v)\|}{\|D^2v\|} ~(=\sup_{\scriptsize{\begin{array}{c}
     v\in V_0(K)  \\
     D^2v\neq0 
\end{array}}}\frac{\|\nabla v\|}{\|D^2v\|})\:.
$$
Thus, the interpolation error estimation of $\Pi_1$ can be given as 
$$
\| v -\Pi_1 v\| \le C_{L,0} \|D^2v\|, \quad  \| \nabla(v -\Pi_1 v)\| \le C_{L,1} \|D^2v\|\:.
$$
The constants are determined by solving the following eigenvalue problems: \\
\textbf{Problem c)}  Find $u\in V_0(K)$ and $\eta>0$ such that
$$
(D^2u,D^2v)=\eta(u,v)\quad\forall v\in V_0(K) \:.
$$
\textbf{Problem d)}  Find $u\in V_0(K)$ and $\mu>0$ such that
$$
(D^2u,D^2v)=\mu (\nabla u, \nabla v)\quad\forall v\in V_0(K) \:.
$$
Denote the smallest eigenvalue of each problem by $\eta_1$ and $\mu_1$, respectively. Then $C_{L,0}=1/\sqrt{\eta_1}$ and $C_{L,1}=1/\sqrt{\mu_1}$.   
Choose the subspace of Fujino--Morley finite element space $V^h$ defined over triangulation of $K$,
$$
V^{h}_0(K)=\{v_h\in V^{h}~|~v_h(P_i)=0,\quad i=1,2,3\}\:.
$$
Let $\eta_{h,1}$ and $\mu_{h,1}$ be the smallest eigenvalues of Problem c) and d), respectively, with $V_0(K)$ replaced by $V_0^h(K)$.

Let us take the following setting for Theorem 2.1 of \cite{Liu2015AMC}, which is similar to Theorem \ref{Liu_thm}, 
\begin{eqnarray*}
V(h):=\{u+u_h~|~u\in V_0,~u_h\in V^h_0\},\quad M(u,v):=(D^2u,D^2v) , \\
N(u,v):=(u,v)~\text{for}~C_{L,0},\quad N(u,v):=(\nabla u,\nabla v)~\text{for}~C_{L,1} \:.
\end{eqnarray*}
With the newly obtained error constant estimation in \S\ref{header-n753}, we obtain the lower bounds of $\eta_1$ and $\mu_1$ like (\ref{framework lower bounds_1}). 
$$
\eta_1\ge\frac{\eta_{h,1}}{1+(0.07353h^2)^2 \eta_{h,1}},
\quad \mu_1\ge\frac{\mu_{h,1}}{1+(0.18868h)^2\mu_{h,1}}\:.
$$

\paragraph{The Fujino--Morley interpolation error constants} 
For the Fujino--Morley interpolation $\Pi^{\FM}$, let us define the following constants
$$
C_{\FM,0} (K)^{-2}= 
\eta (K)=\min_{\scriptsize{\begin{array}{c}
     v\in H^2(K)  \\
     v-\Pi^{\FM} v \not = 0 
\end{array}}} \frac{\|D^2v\|^2}{\|v-\Pi^{\FM}v\|^2}
$$
$$
C_{\FM,1} (K)^{-2}=
\mu (K)=\min_{\scriptsize{\begin{array}{c}
     v\in H^2(K)  \\
     \nabla(v-\Pi^{\FM}v)\not = 0 
\end{array}}}\frac{\|D^2v\|^2}{\| \nabla (v-\Pi^{\FM}v)\|^2}\:.
$$
Notice that the minimizer function for $\eta$ and $\mu$ are both orthogonal to all $P_2$ polynomial functions with respect to $(D^2\cdot, D^2 \cdot)$. 
To see this, one can take the perturbation of the minimizer function with respect to any function in $P_2(K)$.

Let us introduce the space $U_0(K)$ by 
$$
U_0(K):=\{v \in H^2(K)~ |~ (D^2 v, D^2 p)=0, \forall p \in P^2(K) \}
$$
Then $\eta$ and $\mu$ can be characterized by Rayleigh quotients over $U_0(K)$,
$$
\eta (K)=\min_{\scriptsize{\begin{array}{c}
     v\in U_0(K)  \\
     v-\Pi^{\FM} v \not = 0 
\end{array}}} \frac{\|D^2v\|^2}{\|v-\Pi^{\FM}v\|^2},
\quad 
\mu (K)=\min_{\scriptsize{\begin{array}{c}
     v\in U_0(K)  \\
     \nabla(v-\Pi^{\FM}v)\not = 0 
\end{array}}} \frac{\|D^2v\|^2}{\| \nabla(v-\Pi^{\FM}v)\|^2}\:.
$$
Notice that $(D^2\cdot, D^2 \cdot)$ and $((I-\Pi^{FM})\cdot, (I-\Pi^{FM})\cdot  )$
are both positive definite bilinear forms on $U_0(K)$. Thus, we can apply
Theorem 2.1 of \cite{Liu2015AMC} to solve corresponding eigenvalue problems.\\

\paragraph{Concrete values of interpolation error constants}
 In Table \ref{tab:LB_1} and \ref{tab:LB_2}, 
we list the estimation of upper bounds of $C_{L,i}$ and $C_{{\FM},i}$ for different shapes of triangle $K$. 
The underline in tables tells that the lower bound and upper bound evaluation of the constants agree with each other at the underlined digits.

\begin{remark}
The estimation of constants considered in this paper only concerns the largest edge length of a triangle element.
By utilizing more geometric information of the element, i.e., the inner angle size and each edge length, 
one can have better upper bounds of interpolation constants over domain of different shapes; 
see the work of Liu-Kikuchi \cite{Kikuchi+Liu2007,LIU-Kikuchi} and Kobayashi \cite{Kobayashi-2015-CM}.
\end{remark}

\begin{table}[htbp]
    \centering
    \caption{Error constants $C_{L,0}$ and $C_{L,1}$ for the Lagrange interpolation}
    \label{tab:LB_1}
    \begin{tabular}{|p{3cm}||p{2.2cm}|p{2.2cm}| }
 \hline
$(a,b)$ &$C_{L,0}$  & $C_{L,1}$\\
 \hline
$(0,1)$&\underline{0.167}349 &\underline{0.4887}67 \\
$(1/2,\sqrt3/2)$&\underline{0.117}134 &\underline{0.3184}57\\
$(-\sqrt2/2,\sqrt2/2)$&\underline{0.245}388 &\underline{1.1879}98\\
$(0,0.1)$&\underline{0.108}221 &\underline{0.327}955\\
 \hline
\end{tabular}
\end{table}

\begin{table}[htbp]
    \centering
        \caption{Error constants $C_{\FM,0}$ and $C_{\FM,1}$ for the Fujino--Morley interpolation}
    \label{tab:LB_2}
    \begin{tabular}{|p{3cm}||p{2.2cm}|p{2.2cm}| }
 \hline
$(a,b)$ &$C_{\FM,0}$  & $C_{\FM,1}$\\
 \hline
$(0,1)$&\underline{0.090}287 &\underline{0.233}708 \\
$(1/2,\sqrt3/2)$&\underline{0.073}583 &\underline{0.1743}54\\
$(-\sqrt2/2,\sqrt2/2)$&\underline{0.093}318 &\underline{0.300}773\\
$(0,0.1)$&\underline{0.060}474 &\underline{0.188}918\\
 \hline
\end{tabular}

\end{table}

\section{Summary}\label{ch_sum}
In this paper, we provide optimal estimation of two important error constants in bounding eigenvalues of bi-harmonic operators. 
As application of obtained estimation of the constants, the upper bounds for error constants of two interpolation operators are evaluated. 
Moreover, the algorithm proposed here along with the explicit constant values can be further used to
give rigorous bounds for the eigenvalues of general bi-harmonic differential operators.

\begin{acknowledgements}
The authors would like to thank for the support from the Ministry of Science and Technology, Taiwan, R.O.C. under Grant no. MOST 107-2911-M-006-506.
This research is also supported by Japan Society for the Promotion of Science, 
Grand-in-Aid for Young Scientist (B) 26800090,
Grant-in-Aid for Scientific Research (C) 18K03411 and 	
Grant-in-Aid for Scientific Research (B) 16H03950 for the third author.

\end{acknowledgements}

\bibliographystyle{spmpsci}      

\bibliography{reference} 

%
%

\end{document}